\documentclass[12pt,reqno]{amsart}
\usepackage{amsmath,amssymb,amsfonts,amscd,latexsym,amsthm,mathrsfs,verbatim}
\usepackage[unicode]{hyperref}
\newtheorem{lemma}{Lemma}

\newtheorem{proposition}{Proposition}
\theoremstyle{remark}
\newtheorem*{remark}{\bf Remark}

\renewcommand{\d}{{\mathrm d}}

\newcommand{\lcm}{\operatorname{lcm}}
\newcommand{\ord}{\operatorname{ord}}
\newcommand{\bm}{{\boldsymbol m}}

\newcommand{\cM}{{\mathcal M}}
\newcommand{\cP}{{\mathcal P}}
\begin{document}

\title[The irrationality measure of $\pi$ is at most $7.103205334137\dots$]{The irrationality measure of $\boldsymbol\pi$ \\ is at most $\mathbf{7.103205334137\dots}$}

\date{13 December 2019. \emph{Revised}: 7 January 2020}

\author{Doron Zeilberger}
\address{Department of Mathematics, Rutgers University (New Brunswick), Hill Center-Busch Campus, 110 Frelinghuysen Rd., Piscataway, NJ 08854-8019, USA}
\email{doronzeil@gmail.com}

\author{Wadim Zudilin}
\address{Department of Mathematics, IMAPP, Radboud University, PO Box 9010, 6500~GL Nijmegen, Netherlands}
\email{w.zudilin@math.ru.nl}

\subjclass[2010]{Primary 11J82; Secondary 11Y60, 33F10, 33C60}
\keywords{$\pi$; irrationality measure; experimental mathematics; Almkvist--Zeilberger algorithm}

\dedicatory{In memory of Naum Il'ich Feldman (1918--1994)}

\begin{abstract}
We use a variant of Salikhov's ingenious proof that the irrationality measure of $\pi$ is at most $7.606308\dots$
to prove that, in fact, it is at most $7.103205334137\dots$\,.
\\
\textsc{Accompanying Maple package}.
While this article has a fully rigorous human-made and human-readable proof of the claim in the title,
it was \emph{discovered} thanks to the Maple package \texttt{SALIKOHVpi.txt} available from\\
\url{http://sites.math.rutgers.edu/~zeilberg/mamarim/mamarimhtml/pimeas.html}.
\end{abstract}

\maketitle

\part*{Introduction: How irrational is $\pi$?}

Every number that is not rational (a quotient of integers) is \emph{irrational}, but not all
irrational numbers are born equal. To measure `how irrational' is a given number $x$,  
we define (see~\cite{Wei}) the \emph{irrationality measure} $\mu$ (also called the \emph{irrationality exponent})
as the smallest number $\mu$ such that
$$
\biggl| x- \frac{p}{q}\bigg| > \frac{1}{q^{\mu +\epsilon}}
$$
holds for any $\epsilon>0$ and all integers $p$ and $q$ with sufficiently large~$q$.

It is not hard to see that the irrationality measure of $e$ is $2$, but the exact irrationality measure
of $\pi$ is unknown. It became a \textbf{competitive sport} to find
lower and lower upper bounds for the irrationality measure of $\pi$.
The first upper bound, of $42$,  was proved in 1953  by Kurt Mahler \cite{Mah53}.
This record has been subsequently improved by Maurice Mignotte,  Gregory Chudnowsky, and in three better-and-better
articles, by Masayoshi Hata (see the references in \cite{Hat93,Sal10}). The current ``world record'' is due to Vladislav Khasanovich Salikhov who
proved the upper bound of $7.606308$. This was announced \cite{Sal08} in 2008 and published \cite{Sal10} in 2010.
In this article we tweak Salikhov's method to beat his more than ten-year-old record to set a new world record
of $7.103205334137\dots$\,.

Since  the aim of our paper is not \emph{just} to state and prove yet another record that would most likely
be broken again sooner or later (we hope that not that soon, unless it is by ourselves\dots), but
to also explain our ``experimental mathematics'' methodology that 
pointed the way to the ultimate human-generated formal proof, to be given in Part~\ref{partII}.
We also describe a fully rigorous, and fully computer-generated,
proof of a coarser upper bound that is much better than many of the previous world records. This will be
done in Part~\ref{partI}. Readers who are not interested in the process of discovery, or computer proofs, can
go straight to Part~\ref{partII}, that is a self-contained human-generated and human-readable proof.

\medskip
We are grateful to Vladislav Salikhov for pointing out a mistake in the previous version of our Lemma~\ref{lem2} below.
Fixing the gap required from us to employ new techniques, so that at the end this manuscript is more than just tweaking Salikhov's construction in \cite{Sal10}.

\part{The experimental mathematics way}
\label{partI}

\subsection*{General strategy}

A good way to gain immortality, and become a \emph{famous} {\bf person}, is to be the first one to prove that a \emph{famous} {\bf constant},
let's call it $x$, is irrational. 
One way to achieve this is to come up with two
sequences of positive {\bf integers} $\{a_n\}$ and $\{b_n\}$, and a {\bf positive}, explicit real number $\delta$ such that
there is a constant $C$, independent of $n$, such that, for all $n>0$,
$$
\biggl|x- \frac{a_n}{b_n}\biggr| \le \frac{C}{b_n^{1+\delta}}.
$$
This immediately implies the irrationality of $x$ and at the same time establishes an upper bound,
namely  $1+1/\delta$, for the irrationality measure of $x$.

This is exactly how, in 1978, the $64$-year old Roger Ap\'ery became immortal by doing the above with  $x=\zeta(3)$  (i.e., $\sum_{n=1}^{\infty} n^{-3}$);
see Alf van der Poorten's classic exposition~\cite{vdP79}.

Shortly after, Frits Beukers \cite{Beu79} gave a much simpler rendition of Ap\'ery's construction by introducing a certain explicit triple integral
$$
I(n)= \int_0^1 \!\!\!  \int_0^1 \!\!\! \int_0^1
\biggl( \frac{x(1-x)y(1-y)z(1-z)}{1-(1-xy)z)} \biggr)^n \frac{\d x\,\d y\,\d z}{1-(1-xy)z},
$$
and pointing out that
\begin{itemize}
\item[(i)] $I(n)$ is small and can be explicitly bounded,
\item[(ii)] $I(n)=A(n) +B(n) \zeta(3)$ for certain sequences of rational numbers $A(n),B(n)$ that can be explicitly bounded, and
\item[(iii)] $A(n)\lcm(1,2, \dots, n)^3$ and  $B(n)$ are integers.
\end{itemize}
Since thanks to the Prime Number Theorem $\lcm(1,2,\dots,n)$ grows like $e^{n+o(n)}$ as $n\to\infty$, everything followed.

Shortly after \cite{AR79}, Krishna Alladi and Michael Robinson used one-dimensional analogs to reprove the irrationality of $\log 2$,
and established an upper bound of $4.63$ for its irrationality measure (subsequently improved, see \cite{Wei}) by considering the simple integral
$$
I(n)= \int_0^1 \biggl( \frac{x(1-x)}{1+x} \biggr)^n \,\frac{\d x}{1+x}.
$$
Our coming manuscript \cite{ZZ20} is dedicated to further exploration of this theme.

\subsection*{An experimental mathematics redux of Salikhov's approach}

Salikhov \cite{Sal10} essentially uses the same strategy, but with the far more complicated integral
$$
I(n)=-i\int_{4-2i}^{4+2i}\biggl(\frac{(x-4+2i)^6(x-4-2i)^6(x-5)^6(x-6+2i)^6(x-6-2i)^6}
{x^{10} (x-10)^{10}}\biggr)^n\,\frac{\d x}{x(x-10)}.
$$
He then used {\bf partial fractions} to claim that
$$
I(n)=A(n) + B(n) \pi,
$$
for some sequences of $\{A(n)\} \,, \,\{B(n)\}$ of \emph{rational numbers}. 

Using the \textbf{saddle-point method}, he bounded $I(n)$ and $A(n)$, $B(n)$.

He then proved that if one sets
$$
A'(n) =  \lcm(1,2,\dots, 10n) \biggl(\frac{25}{32}\biggr)^nA(n)
\quad\text{and}\quad 
B'(n) =  \lcm(1,2,\dots, 10n) \biggl(\frac{25}{32}\biggr)^nB(n),
$$
then $A'(n)$ and $B'(n)$ are {\bf integer sequences}, and defining
$$
I'(n) = \lcm(1,2,\dots,10n) \biggl(\frac{25}{32}\biggr)^n I(n),
$$
using $\lcm(1,2,\dots,10n)=O(e^{10n+o(n)})$, one can explicitly bound $A'(n),B'(n),I'(n)$,
and $I'(n)$ being small and $B'(n)$ being big one can get a crude upper bound for the irrationality measure,
using the fact that $A'(n)/B'(n)$  approximate  $\pi$.

Finally, the hard part was to come up with `additional saving', a sequence of integers $F(n)$, 
such that $A''(n)=A'(n)/F(n)$ and $B''(n)=B'(n)/F(n)$ are still integers.
Setting $I''(n)=I'(n)/F(n)$ he squeezed more
juice out of it, getting a larger $\delta$ and hence a smaller irrationality
measure $1+1/\delta$, setting the current record of $7.606308\dots$\,.

Our approach is different. We do not use partial fractions, but rather the fact, that thanks to 
the Almkvist--Zeilberger algorithm \cite{AZ90}, there exists a third-order linear recurrence equation of the form
$$
p_0(n) I(n)+ p_1(n) I(n+1) + p_2(n) I(n+2) + p_3(n) I(n+3) = 0 , 
$$
for some explicit polynomials $p_0(n), p_1(n), p_2(n), p_3(n)$. 
To save space, we do not reproduce it here, but refer the reader to the following webpage:\\
\url{http://sites.math.rutgers.edu/~zeilberg/tokhniot/oSALIKHOVpi2.txt}.

That web-page gives a new, computer-generated proof of the crude upper bound, only using the
recurrence and the so-called Poincar\'e lemma that gives the asymptotics of $A(n)$, $B(n)$ and $I(n)$ from which
it is immediate to bound $A'(n)$, $B'(n)$, and $I'(n)$. The only non-rigorous part in our
approach is the  study of the extra divisor $F(n)$, whose growth we estimate empirically.

For details see the above-mentioned computer-generated  article.

\subsection*{Tweaking Salikhov's integral. Looking where to dig}

Looking at Salikhov's integral, it is natural to consider the more general integral
$$
I_{A,B}(n) =
-i\int_{4-2i}^{4+2i}\frac{(x-4+2i)^{2An}(x-4-2i)^{2An}(x-5)^{2An}(x-6+2i)^{2An}(x-6-2i)^{2An}}
{x^{2Bn+1} (x-10)^{2Bn+1}}\,\d x,
$$
where Salikhov's integral is the special case $I_{3,5}(n)$. Perhaps we can do better? But before we invest
time and energy, trying out many choices of $A$ and $B$, it makes sense to
do things \emph{empirically}, crank out, say, 300 terms of the examined sequence and see whether it yields good `deltas'.

Alas, even Maple and Mathematica will start to complain if we use the definition for, say $n=300$. Luckily, for
each specific $A$ and $B$, Shalosh B.~Ekhad can quicky use the Almkvist--Zeilberger algorithm \cite{AZ90} to crank out
many terms, and thereby get very good estimates for the `deltas'. This initial \emph{reconnaissance} is very fast and
gives you an indication on \emph{where to dig}.

This is implemented in procedure \texttt{BestAB} in the Maple package \texttt{SALIKOHVpi.txt} mentioned above.
Typing \texttt{BestAB(10,300);} gives the following computer-generated article:\\
\url{http://sites.math.rutgers.edu/~zeilberg/tokhniot/oSALIKHOVpi4.txt}.

Most of the choices of $(A,B)$ give negative, useless, deltas, but\,---\,{\bf surprise!}\,---\,the choice of
$A=2$, $B=3$  yielded  that the smallest $\delta$ in the range $290 \leq n \leq 300$ was $0.16605428729395818514$.
This beats the analogous value for the $A=3$, $B=5$ case, that equals $0.15727140930557009691$.
The `bronze medal' was won by $A=5$, $B=8$ that was almost as good: $0.15701995819256081077$;
followed by $A=8$, $B=13$ that gave the respectable $0.15586354092162189848$.
Next in line was a non-Fibonacci $A=7$, $B=10$ that placed fifth, with  $0.12451550531454231901$.
For all the other `empirical deltas' see the above output file.

Once we found out that $A=2$, $B=3$ was a good gamble, we had another pleasant surprise.
We can replace $n$ by $n/2$ and still get combinations of $1$ and $\pi$
(in the original case $A=3$, $B=5$ of Salikhov, the odd indices $n$ give combinations of $1$ and $\arctan(1/7)$).
This simplifies the recurrence, and a fully rigorous proof of the cruder upper bound of $10.747747465671804677\dots$ can be found here:\\
\url{http://sites.math.rutgers.edu/~zeilberg/tokhniot/oSALIKHOVpi3.txt}.

In order to get the more refined upper bound, we had to resort to non-rigorous estimates.

Luckily it was possible to make everything fully rigorous, and this brings us to Part~\ref{partII}.

\part{A fully rigorous (human-generated) proof \\ of the claimed upper bound for the irrationality measure of~$\pi$}
\label{partII}

\subsection*{Test bunny}
For $n=0,1,2,\dots$, our integrals in question are
\begin{align}
I_n
&=5i\int_{4-2i}^{4+2i}\frac{(x-5)^{2n}(x-4+2i)^{2n}(x-4-2i)^{2n}(x-6+2i)^{2n}(x-6-2i)^{2n}}
{x^{3n+1}(x-10)^{3n+1}}\,\d x
\nonumber\\
&=i(-1)^{n+1}\int_{-1-2i}^{-1+2i}\frac{5x^{2n}(x+1+2i)^{2n}(x+1-2i)^{2n}(x-1+2i)^{2n}(x-1-2i)^{2n}}
{(5+x)^{3n+1}(5-x)^{3n+1}}\,\d x.
\label{eq:1}
\end{align}
These are from the winning family in Part~\ref{partI}.

\subsection*{Arithmetic}
The integrand
\begin{align}
R(x)
&=\frac{5x^{2n}(x+1+2i)^{2n}(x+1-2i)^{2n}(x-1+2i)^{2n}(x-1-2i)^{2n}}
{(5+x)^{3n+1}(5-x)^{3n+1}}
\nonumber\\
&=\frac{5x^{2n}(x^4+6x^2+25)^{2n}}{(5+x)^{3n+1}(5-x)^{3n+1}}
\label{eq:2}
\end{align}
possesses the symmetry $R(-x)=R(x)$ and therefore can be written as
\begin{equation}
R(x)=P(x)+\sum_{j=0}^{3n}\biggl(\frac{A_j}{(5+x)^{j+1}}+\frac{A_j}{(5-x)^{j+1}}\biggr)
\label{eq:3}
\end{equation}
for some rational $A_j$ and a polynomial $P(x)\in\mathbb Z[x^2]$ of degree $4n-2$.

\begin{lemma}
\label{lem1}
The coefficients $A_j$ in the partial-fraction expansion \eqref{eq:3} satisfy
\begin{equation}
2^{-\lfloor(5n+3j)/2\rfloor+1}5^{-j}A_j\in\mathbb Z
\quad\text{for}\; j=0,1,\dots,3n.
\label{eq:inc}
\end{equation}
In particular, they are integers.
\end{lemma}

\begin{proof}
To compute $A_j$, introduce linear operators
$$
D_m\colon f(x)\mapsto\frac1{m!}\,\frac{\d^mf(x)}{\d x^m}\bigg|_{x=-5}.
$$
Then with the help of Leibniz's formula we deduce
\begin{align}
A_j
&=D_{3n-j}\bigl((x+5)^{3n+1}R(x)\bigr)
\displaybreak[2]\nonumber\\
&=5\sum_{\substack{m_0,m_1,\dots,m_5\ge0 \\ m_1,\dots,m_5\le 2n \\ m_0+m_1+\dots+m_5=3n-j}}
D_{m_0}(5-x)^{-3n-1}
D_{m_1}x^{2n}D_{m_2}(x+1+2i)^{2n}D_{m_3}(x+1-2i)^{2n}
\nonumber\\[-21pt] &\qquad\qquad\qquad\qquad\qquad\qquad\qquad\times
D_{m_4}(x-1+2i)^{2n}D_{m_5}(x-1-2i)^{2n}
\displaybreak[2]\nonumber\\
&=5\sum_{\bm\in\cM_j}(-1)^{m_1+\dots+m_5}T(\bm)10^{-3n-1-m_0}5^{2n-m_1}
(4-2i)^{2n-m_2}(4+2i)^{2n-m_3}
\nonumber\\[-6pt] &\qquad\qquad\qquad\qquad\qquad\times
(6-2i)^{2n-m_4}(6+2i)^{2n-m_5}
\displaybreak[2]\nonumber\\
&=\sum_{\bm\in\cM_j}(-1)^{m_1+\dots+m_5}T(\bm)
2^{4n-1+j+m_1}(1-i)^{-m_4}(1+i)^{-m_5}
\nonumber\\[-6pt] &\qquad\qquad\qquad\qquad\times
5^j(2+i)^{m_2+m_5}(2-i)^{m_3+m_4}
\label{eq:A_j}
\end{align}
for $j=0,\dots,3n$, where the summation is over the multi-indices $\bm=(m_0,\dots,m_5)$ from the set
\begin{align}
\cM_j
=\{(m_0,m_1,\dots,m_5):\,
& m_0,m_1,\dots,m_5\ge0;\ m_1,\dots,m_5\le 2n;\
\nonumber\\ &\quad
m_0+m_1+\dots+m_5=3n-j\}
\subset\mathbb Z_{\ge0}^6
\nonumber
\end{align}
and
\begin{equation}
T(\bm)=T(m_0,m_1,\dots,m_5)
=\binom{3n+m_0}{m_0}\prod_{\ell=1}^5\binom{2n}{m_\ell}\in\mathbb Z.
\nonumber
\end{equation}
Now $m_4+m_5\le3n-j$ and $(1\pm i)^2=\pm2i$; hence
$$
2^{\lceil(3n-j)/2\rceil}\times(1-i)^{-m_4}(1+i)^{-m_5}\in\mathbb Z[i]
$$
and
$$
2^{-\lfloor(5n+3j)/2\rfloor+1}\times2^{4n-1+j+m_1}(1-i)^{-m_4}(1+i)^{-m_5}\in\mathbb Z[i].
$$
Therefore, $2^{-\lfloor(5n+3j)/2\rfloor+1}5^{-j}A_j\in\mathbb Z[i]$ and the result follows from using the fact
that $A_j\in\mathbb Q$.
\end{proof}

Formula \eqref{eq:A_j} for the coefficients $A_j$ makes sense for \emph{any} integer $j\le3n$;
it generates the coefficients in the Laurent series expansion of $R(x)$ at $x=-5$. More precisely,
\begin{equation*}
R(x)
=\sum_{k=-3n}^\infty A_{-k}(x+5)^{k-1}
=\sum_{j=0}^{3n}\frac{A_j}{(x+5)^{j+1}}
+\sum_{k=1}^\infty A_{-k}(x+5)^{k-1}.
\end{equation*}
Note that $A_j$ produced by formula \eqref{eq:A_j} are not necessarily integral for \emph{negative} $j$ but at least they satisfy $10^{-j}A_j\in\mathbb Z$ for $j=-1,-2,\dots,-(4n-1)$ on the grounds of the formula; and we also have $10^{-j}A_j\in\mathbb Z$ for $j=0,1,2,\dots,3n$ in accordance with Lemma~\ref{lem1}.
Furthermore,
$$
\sum_{j=0}^{3n}\frac{A_j}{(5-x)^{j+1}}
=\sum_{j=0}^{3n}\frac{A_j}{(10-(x+5))^{j+1}}
=\sum_{j=0}^{3n}A_j\sum_{k=1}^\infty\binom{j+k-1}j\frac{(x+5)^{k-1}}{10^{j+k}};
$$
comparing the last two expansions with \eqref{eq:3} we find out that
\begin{align*}
P(x)
&=R(x)-\sum_{j=0}^{3n}\biggl(\frac{A_j}{(5+x)^{j+1}}+\frac{A_j}{(5-x)^{j+1}}\biggr)
\\
&=\sum_{k=1}^\infty\biggl(A_{-k}
-\sum_{j=0}^{3n}\binom{j+k-1}j\frac{A_j}{10^{j+k}}\biggr)(x+5)^{k-1}.
\end{align*}
On the other hand, $P(x)$ is a \emph{polynomial} of degree $4n-2$, hence
\begin{equation}
P(x)
=\sum_{k=1}^{4n-1}\biggl(A_{-k}
-\sum_{j=0}^{3n}\binom{j+k-1}j\frac{A_j}{10^{j+k}}\biggr)(x+5)^{k-1}.
\label{eq:x+5}
\end{equation}

\begin{lemma}
\label{lem2}
Any prime from the set
$$
\cP_n=\biggl\{p>\max\{5,\sqrt{3n}\}:\frac12\le\Bigl\{\frac np\Bigr\}<\frac23\biggr\}
\subset\{p\; \text{prime}:5<p\le2n\}
$$
satisfies the following property\textup: if $p\mid j$ for $j\in\{-4n+1,-4n+2,\dots,3n\}$,
then $A_j\equiv0\pmod p$ \textup(in other words, $p\mid 10^{-j}A_j$\textup).
\textup{(Here $\{x\}=x-\lfloor x\rfloor$ denotes the fractional part of the number.)}
\end{lemma}

\begin{proof}
In order to establish the claim, we will cast the coefficients $A_j$ in \eqref{eq:A_j} differently.
Observe that
\begin{align*}
R(x)
&=\frac{5x^{2n}(x^2+(3-4i))^{2n}(x^2+(3+4i))^{2n}}{(25-x^2)^{3n+1}}
\\
&=5\sum_{n_1,n_2\ge0}\binom{2n}{n_1}\binom{2n}{n_2}(3-4i)^{2n-n_1}(3+4i)^{2n-n_2}
\frac{x^{2(n+n_1+n_2)}}{(25-x^2)^{3n+1}}
\end{align*}
and
\begin{align*}
\frac{x^{2m}}{(25-x^2)^{3n+1}}
&=\frac{(5-(x+5))^{2m}}{(x+5)^{3n+1}(10-(x+5))^{3n+1}}
\\
&=\frac{5^{2m}10^{-3n-1}}{(x+5)^{3n+1}}
\sum_{k=0}^\infty\frac{(x+5)^k}{10^k}\sum_{n_0\ge0}(-2)^{n_0}\binom{2m}{n_0}\binom{3n+k-n_0}{3n},
\end{align*}
hence
\begin{align*}
A_j
&=\sum_{n_1,n_2\ge0}(3-4i)^{2n-n_1}(3+4i)^{2n-n_2}5^{2n+2n_1+2n_2+1}10^{-(6n-j)-1}
\\ &\qquad\times
\binom{2n}{n_1}\binom{2n}{n_2}Z(n,n_1+n_2,j),
\end{align*}
where
$$
Z(n,m,j)=\sum_{n_0\ge0}(-2)^{n_0}\binom{2n+2m}{n_0}\binom{6n-j-n_0}{3n}.
$$
This means that our lemma is a consequence of the following divisibility property:
If a prime $p\in\cP_n$ divides $j$, then it also divides
$$
\hat T(n,n_1,n_2)=\binom{2n}{n_1}\binom{2n}{n_2}Z(n,n_1+n_2,j)
$$
for any $n_1,n_2\ge0$.

From now on, we will repeatedly use the fact that the $p$-adic order of $N!$ satisfies
$\ord_pN!=\lfloor N/p\rfloor=N/p-\{N/p\}$ when $p>\sqrt N$. In particular,
\begin{equation}
\ord_p\binom{2n}{n_\ell}=\lfloor2\omega\rfloor-\lfloor2\omega-\omega_\ell\rfloor-\lfloor\omega_\ell\rfloor
=\lfloor2\omega\rfloor-\lfloor2\omega-\omega_\ell\rfloor
\quad\text{for}\;\ell=1,2,
\label{eq:pord}
\end{equation}
where the fractional parts $\omega=\{n/p\}$, $\omega_1=\{n_1/p\}$ and $\omega_2=\{n_2/p\}$ all belong to the interval $[0,1)$.
Since $p\in\cP_n$, we have $\omega\in[\frac12,\frac23)$, so that $\lfloor2\omega\rfloor=\lfloor3\omega\rfloor=1$.
If at least one of the $p$-adic orders in \eqref{eq:pord} is positive then immediately $\ord_p\hat T(n,n_1,n_2)\ge1$ establishing the required divisibility; therefore, it remains to analyze the remaining situations assuming
$\lfloor2\omega-\omega_\ell\rfloor=\lfloor2\omega\rfloor=1$
for $\ell=1,2$;
in other words, assuming
\begin{equation*}
2\omega-\omega_1\ge1 \quad\text{and}\quad 2\omega-\omega_2\ge1.
\end{equation*}

The binomial sums $Z(n,m,j)$ can be realized as a terminating $_2F_1$ hypergeometric function,
to which several classical transformations can be applied. For example, it can be transformed into
\begin{align*}
Z(n,m,j)
&=\sum_{n_0\ge0}(-1)^{n_0}\binom{2n+2m}{n_0}\binom{6n-2(n+m)-j}{3n-j-n_0}
\\
&=(-1)^{n+m}\sum_{k\in\mathbb Z}(-1)^k\binom{2n+2m}{n+m+k}\binom{4n-2m-j}{2n-m-k}.
\end{align*}
Though the expression does not possess a closed form in general, its particular instance $j=0$ reduces to the \emph{super Catalan numbers}
\begin{equation*}
\frac{(2N)!\,(2M)!}{N!\,(N+M)!\,M!}
=\sum_{k=-\infty}^{\infty}(-1)^k\binom{2N}{N+k}\binom{2M}{M+k};
\end{equation*}
see \cite{Ges92}, also for the historical reference of this identity due to K.~von Szily (1894).
The argument in \cite[Sect.~6]{Ges92} shows that the more general sum
$$
\sum_{k=-\infty}^{\infty}(-1)^k\binom{2N}{N+k}\binom{2M-j}{M+k}
$$
is the coefficient of $t^{2N}$ in the polynomial
$$
(-1)^N\frac{(2N)!\,(2M-j)!}{(N+M)!\,(N+M-j)!}\,(1+t)^{N+M}(1-t)^{N+M-j}.
$$
In our situation $N=n+m$, $M=2n-m$ with $m=n_1+n_2$, the factorial-ratio factor
$$
\frac{(2N)!\,(2M-j)!}{(N+M)!\,(N+M-j)!}
=\frac{(2n+2n_1+2n_2)!\,(4n-2n_1-2n_2-j)!}{(3n)!\,(3n-j)!}
$$
has the nonnegative $p$-adic order
\begin{align*}
&
\lfloor2\omega+2\omega_1+2\omega_2\rfloor
+\lfloor4\omega-2\omega_1-2\omega_2-j/p\rfloor
-\lfloor3\omega\rfloor-\lfloor3\omega-j/p\rfloor
\\ &\quad
=\lfloor2\omega+2\omega_1+2\omega_2\rfloor
+\lfloor4\omega-2\omega_1-2\omega_2\rfloor
-2\lfloor3\omega\rfloor
\end{align*}
(we use $j/p\in\mathbb Z$), because $\lfloor3\omega\rfloor=1$,
$$
2\omega+2\omega_1+2\omega_2\ge2\omega\ge1
\quad\text{and}\quad
4\omega-2\omega_1-2\omega_2\ge4\omega-4(2\omega-1)=4(1-\omega)>\frac43.
$$
Moreover, if this $p$-adic order is \emph{positive} then $Z(n,n_1+n_2,j)$ is divisible by $p$, hence the divisibility of $\hat T(n,n_1,n_2)$ follows. Thus, we are left with the situation when this order is zero,
$$
\lfloor2\omega+2\omega_1+2\omega_2\rfloor=\lfloor4\omega-2\omega_1-2\omega_2\rfloor=1,
$$
meaning that
\begin{equation}
2\omega+2\omega_1+2\omega_2<2 \quad\text{and}\quad 4\omega-2\omega_1-2\omega_2<2.
\label{eq:ineq2}
\end{equation}
We have to show that the coefficient of $t^{2N}$ in $(1+t)^{N+M}(1-t)^{N+M-j}$ is divisible by $p$.

Denoting $r=-j/p\in\mathbb Z$ and using the ``Freshman's Dream Identity'' 
$(1-t)^p\equiv1-t^p\pmod p$  in the ring $\mathbb Z[[t]]$, we find out that
\begin{align*}
&
(1+t)^{N+M}(1-t)^{N+M-j}
=(1-t^2)^{N+M}(1-t)^{-j}
\\ &\quad
\equiv(1-t^2)^{N+M}(1-t^p)^r
=\sum_{k_1\ge0}(-1)^{k_1}\binom{N+M}{k_1}t^{2k_1}
\sum_{k_2\ge0}(-1)^{k_2}\binom r{k_2}t^{pk_2},
\end{align*}
hence the coefficient of $t^{2N}$ is congruent to
$$
\sum_{k=0}^{\lfloor N/p\rfloor}(-1)^{k+N}\binom{N+M}{N-kp}\binom r{2k}
$$
modulo $p$. The $p$-adic order of the \emph{nonzero} binomial coefficients $\binom{N+M}{N-kp}$ does not depend on~$k$:
\begin{align*}
\ord_p\binom{N+M}{N-kp}
&=-\biggl\{\frac Np+\frac Mp\biggr\}+\biggl\{\frac Np-k\biggr\}+\biggl\{\frac Mp+k\biggr\}
\\
&=-\biggl\{\frac Np+\frac Mp\biggr\}+\biggl\{\frac Np\biggr\}+\biggl\{\frac Mp\biggr\}.
\end{align*}
Recalling that $N=n+m$, $M=2n-m$ with $m=n_1+n_2$ the latter quantity reads
$$
\ord_p\binom{3n}{n+n_1+n_2}
=\lfloor3\omega\rfloor-\lfloor\omega+\omega_1+\omega_2\rfloor-\lfloor2\omega-\omega_1-\omega_2\rfloor=1,
$$
where we employed \eqref{eq:ineq2} to get
$$
\lfloor\omega+\omega_1+\omega_2\rfloor
=\lfloor2\omega-\omega_1-\omega_2\rfloor
=0.
$$
This means that all binomial coefficients  $\binom{N+M}{N-kp}$ are divisible by $p$, thus completing our proof of the divisibility of $\hat T(n,n_1,n_2)$ by $p$, and of the lemma.
\end{proof}

\begin{remark}
An earlier version of the lemma claimed that any prime $p\in\cP_n$, $p\mid j$ for $j\in\{-4n,-4n+1,\dots,3n\}$,
divides $T(\bm)$ for all $\bm\in\cM_j$; this would clearly imply the present statement in view of formula~\eqref{eq:A_j}.
However, the claim about the divisibility properties of $T(\bm)$ was false.
\end{remark}

\begin{lemma}
\label{lem3}
Define $\Phi=\Phi_n=\prod_{p\in\cP_n}p$ and
\begin{equation}
L_n=\frac{\lcm(1,2,\dots,4n)}{\Phi_n}\in\mathbb Z.
\nonumber
\end{equation}
Then
\begin{equation}
L_n\times\frac{10^{-j}A_j}j\in\mathbb Z
\quad\text{for}\; j\in\{-4n,-4n+1,\dots,3n-1,3n\}, \; j\ne0,
\label{eq:inc1}
\end{equation}
and $\Phi_n^{-1}\times A_0\in\mathbb Z$.
\end{lemma}

Asymptotically,
\begin{equation}
\lim_{n\to\infty}\frac{\log\Phi_n}n
=\frac{\Gamma'(2/3)}{\Gamma(2/3)}-\frac{\Gamma'(1/2)}{\Gamma(1/2)}
=\frac\pi{2\sqrt3}-\log\frac{3\sqrt3}4
=0.64527561\dotsc
\label{eq:Phi}
\end{equation}
(see \cite[Lemma 2.2]{Hat93}).

\begin{proof}
Note that
$$
\lcm(1,2,\dots,4n)\times\frac1j\in\mathbb Z
\quad\text{for}\; j\in\{-4n,-4n+1,\dots,3n\}, \; j\ne0,
$$
implying, for all such~$j$,
\begin{equation}
L_n\cdot\frac1{j/p}\in\mathbb Z
\quad\text{if}\; p\mid j, \; p\in\cP_n.
\label{eq:5}
\end{equation}
On the other hand, it follows from formula \eqref{eq:A_j} and Lemma~\ref{lem2} that
\begin{equation}
\frac{10^{-j}A_j}p\in\mathbb Z
\quad\text{if}\; p\mid j, \; p\in\cP_n.
\label{eq:6}
\end{equation}
Combining \eqref{eq:5} and \eqref{eq:6} results in claim~\eqref{eq:inc1}.
\end{proof}

\begin{lemma}
\label{lem4}
Write the polynomial $P(x)\in\mathbb Z[x]$ in the decomposition \eqref{eq:3} as
\begin{equation}
P(x)=\sum_{k=0}^{4n-2}B_k(x+1+2i)^k
\qquad\text{with}\quad B_k\in\mathbb Z[i] \quad\text{for}\; k=0,1,\dots,4n-2.
\label{eq:P}
\end{equation}
Then
\begin{equation}
2^{-\lfloor5n/2\rfloor+\lceil3k/2\rceil+2}\times B_k\in\mathbb Z[i]
\quad\text{for}\; k=0,1,\dots,4n-2.
\label{eq:inc2}
\end{equation}
\end{lemma}

\begin{proof}
If $k\ge2n$ then $-\lfloor5n/2\rfloor+\lceil3k/2\rceil+2\ge0$ and the inclusion in \eqref{eq:inc2} follows from $B_k\in\mathbb Z[i]$.
Therefore, we only need to verify \eqref{eq:inc2} for $k<2n$;
since $R(x)$ from \eqref{eq:2} has a zero of order $2n$ at $x=-1-2i$, we deduce from~\eqref{eq:3} that
\begin{align*}
B_k
&=-\frac1{k!}\,\frac{\d^k}{\d x^k}\sum_{j=0}^{3n}\biggl(\frac{A_j}{(5+x)^{j+1}}+\frac{A_j}{(5-x)^{j+1}}\biggr)\bigg|_{x=-1-2i}
\\
&=-\sum_{j=0}^{3n}(-1)^k\binom{j+k}k\biggl(\frac{A_j}{(5+x)^{j+k+1}}+(-1)^{j+1}\frac{A_j}{(x-5)^{j+k+1}}\biggr)\bigg|_{x=-1-2i}
\\
&=-\sum_{j=0}^{3n}\binom{j+k}k\biggl(\frac{(-1)^kA_j}{(2(2-i))^{j+k+1}}+\frac{A_j}{(2(1+i)(2-i))^{j+k+1}}\biggr)\end{align*}
for $k=0,1,\dots,2n-1$. It follows then from \eqref{eq:inc} that
\begin{equation*}
2^{-\lfloor5n/2\rfloor+\lceil3k/2\rceil+2}(2-i)^{3n+k+1}\times B_k\in\mathbb Z[i]
\end{equation*}
and, again, we recall $B_k\in\mathbb Z[i]$ to conclude with~\eqref{eq:inc2} for $k<2n$.
\end{proof}

\begin{lemma}
\label{lem5}
For the polynomial $P(x)$ in the decomposition \eqref{eq:3}, we have
\begin{equation}
2^{-\lfloor5n/2\rfloor}L_n\times i\int_{-1-2i}^{-1+2i}P(x)\,\d x\in\mathbb Z.
\label{eq:A}
\end{equation}
\end{lemma}

\begin{proof}
We first compute the integral using representation~\eqref{eq:P}:
\begin{align*}
i\int_{-1-2i}^{-1+2i}P(x)\,\d x
&=i\sum_{k=0}^{4n-2}B_k\int_{-1-2i}^{-1+2i}(x+1+2i)^k\,\d x
\displaybreak[2]\\
&=i\sum_{k=0}^{4n-2}\frac{B_k}{k+1}(4i)^{k+1}
=-\sum_{k=0}^{4n-2}\frac{2^{2k+2}B_k}{k+1}\,i^k
\end{align*}
implying
\begin{equation}
2^{-\lfloor5n/2\rfloor}\lcm(1,2,\dots,4n)\times i\int_{-1-2i}^{-1+2i}P(x)\,\d x\in\mathbb Z[i]
\label{eq:A1}
\end{equation}
on the basis of Lemma~\ref{lem4}. On the other hand, if representation \eqref{eq:x+5} is applied then
\begin{align*}
i\int_{-1-2i}^{-1+2i}P(x)\,\d x
&=i\sum_{k=1}^{4n-1}\biggl(A_{-k}-\sum_{j=0}^{3n}\binom{j+k-1}j\frac{A_j}{10^{j+k}}\biggr)
\int_{-1-2i}^{-1+2i}(x+5)^{k-1}\,\d x
\\
&=i\sum_{k=1}^{4n-1}\biggl(\frac{A_{-k}}k
-\sum_{j=0}^{3n}\binom{j+k-1}j\frac{A_j}{k\,10^{j+k}}\biggr)
\bigl((4+2i)^k-(4-2i)^k\bigr)
\displaybreak[2]\\
&=\sum_{k=1}^{4n-1}\biggl(\frac{A_{-k}}k-\frac{A_0}k\,\frac1{10^k}
-\sum_{j=1}^{3n}\binom{j+k-1}{j-1}\frac{A_j}{j}\,\frac{1}{10^{j+k}}\biggr)
\\ &\qquad\times
2^{k+1}\sum_{\substack{\ell=0\\\ell\;\text{odd}}}^k\binom k\ell(-1)^{(\ell+1)/2}2^{k-\ell}
\end{align*}
is a rational number satisfying
\begin{equation}
\frac{\lcm(1,2,\dots,4n)}{\Phi_n}\times i\int_{-1-2i}^{-1+2i}P(x)\,\d x\in10^{-4n}\mathbb Z
\label{eq:A2}
\end{equation}
on the basis of Lemma~\ref{lem3}. Finally, the two inclusions \eqref{eq:A1} and \eqref{eq:A2} combine into~\eqref{eq:A}.
\end{proof}

\begin{lemma}
\label{lem6}
For the partial-fraction part in \eqref{eq:3}  \textup(without the $j=0$ term\textup), we have
\begin{equation}
2^{-\lfloor5n/2\rfloor+1}L_n\times i\int_{-1-2i}^{-1+2i}
\sum_{j=1}^{3n}A_j\biggl(\frac1{(5+x)^{j+1}}+\frac1{(5-x)^{j+1}}\biggr)\,\d x\in\mathbb Z.
\nonumber
\end{equation}
\end{lemma}

\begin{proof}
This follows from
\begin{align*}
&
i\sum_{j=1}^{3n}A_j\int_{-1-2i}^{-1+2i}\biggl(\frac1{(5+x)^{j+1}}+\frac1{(5-x)^{j+1}}\biggr)\,\d x
\\ &\quad
=i\sum_{j=1}^{3n}\frac{A_j}j
\biggl(\frac1{(4-2i)^j}-\frac1{(4+2i)^j}-\frac1{(6+2i)^j}+\frac1{(6-2i)^j}\biggr)
\\ &\quad
=i\sum_{j=1}^{3n}\frac{A_j}j
\biggl(\frac{(2+i)^j}{2^j5^j}-\frac{(2-i)^j}{2^j5^j}-\frac{(2+i)^j}{2^j(1+i)^j5^j}+\frac{(2-i)^j}{2^j(1-i)^j5^j}\biggr)
\in\mathbb Q
\end{align*}
and the inclusions of Lemma~\ref{lem1} and~\ref{lem3}.
\end{proof}

Lemma~\ref{lem1} and the integrality of $L_n$ imply that $2^{-\lfloor5n/2\rfloor+1}L_n\times A_0\in\mathbb Z$;
together with the calculation
\begin{align*}
\int_{-1-2i}^{-1+2i}\biggl(\frac1{5+x}+\frac1{5-x}\biggr)\d x
&=\log(4+2i)-\log(4-2i)-\log(6-2i)+\log(6+2i)
\\
&=\frac{\pi i}2
\end{align*}
and Lemmas~\ref{lem5}, \ref{lem6} we are thus led to the following statement.

\begin{proposition}
\label{prop1}
For the integrals $I_n$ in \eqref{eq:1}, we have
\begin{equation}
2^{-\lfloor5n/2\rfloor+2}L_n\times I_n\in\mathbb Z+\mathbb Z\pi.
\nonumber
\end{equation}
\end{proposition}

\subsection*{Asymptotics}
By now we have legally settled down that $I_n=a_n+b_n\pi$ for some \emph{rational} $a_n$ and $b_n$.

\begin{proposition}
\label{prop2}
The asymptotics of the integrals $I_n$ and the coefficients $b_n$ in the representation $I_n=a_n+b_n\pi$ is as follows:
$$
\limsup_{n\to\infty}|I_n|^{1/n}=|N_1|=0.029458495928\dots
\quad\text{and}\quad
\lim_{n\to\infty}b_n^{1/n}=N_3=21851.691396\dots,
$$
where
$$
N_{1,2}=0.02930189\hdots \pm i\,0.00303351\dots,
\quad
N_3=21851.691396\dots
$$
are the zeros of polynomial
\begin{equation}
108N^3-2359989N^2+138304N-2048.
\label{eq:ind}
\end{equation}
\end{proposition}

\begin{proof}
This rigorously follows from the Poincar\'e lemma supplied by the rigorously produced\,---\,thanks to the Almkvist--Zeilberger method \cite{AZ90}\,---\,difference equation for the integrals $I_n$ (hence also for $a_n,b_n$), 
whose indicial polynomial (more precisely: the indicial polynomial of its `constant-coefficients approximation') is exacty~\eqref{eq:ind}.
Observe that $|I_n|\le1$ follows from integrating over the \emph{line} interval $[-1-2i,-1+2i]$ and trivially bounding the absolute value of the integrand on it.
However, those who prefer traditional analytical methods can have fun going through the glorious details of the saddle-point method, at least after the change of variables $y=x^2$ is performed in \eqref{eq:1}.
For that one deals with the function
$$
\tilde R(y)=\frac{5g(y)^n}{y-25}, \quad\text{where}\; g(y)=\frac{y(y^2+6y+25)^2}{(y-25)^3},
$$
and with the zeros
$$
y_{1,2}=-1.91975076\hdots \mp i\,1.01250889\dots,
\quad
y_3=66.33950152\dots
$$
of
$$
\frac{g'(y)}{g(y)}=\frac{2y^3-125y^2-500y-625}{y(y^2+6y+25)(y-25)}.
$$
Then $N_j=g(y_j)$ for $j=1,2,3$. The remaining part is performing a suitable deformation of path in \eqref{eq:1} to pass through the saddle points $\sqrt{y_1}$ and $\sqrt{y_2}$ (with the choice of branch such the real parts of the roots are negative) and writing a Cauchy integral for $b_n$ over a closed contour passing through the saddle points $\pm\sqrt{y_3}$.
\end{proof}

\subsection*{World record}
It follows from Propositions~\ref{prop1} and \ref{prop2} that the forms
$$
I_n'=2^{-\lfloor5n/2\rfloor+2}L_nI_n=a_n'+b_n'\pi,
\quad\text{where}\; n=0,1,2,\dots,
$$
all have integral coefficients $a_n',b_n'$ and the asymptotics
\begin{align*}
\limsup_{n\to\infty}\frac{\log|I_n'|}n
&=\log|N_1|-\frac52\,\log2+4-\frac\pi{2\sqrt3}+\log\frac{3\sqrt3}4
=-1.90291648559998\dots
\\ \intertext{and}
\lim_{n\to\infty}\frac{\log b_n}n
&=\log N_3-\frac52\,\log2+4-\frac\pi{2\sqrt3}+\log\frac{3\sqrt3}4
=11.613890045331\dots
\end{align*}
(the asymptotics of $L_n$ follows from the Prime Number Theorem and \eqref{eq:Phi}).
This implies (see, e.g., \cite[Lemma 1]{Sal10}) that the irrationality measure of $\pi$ is bounded above by
$$
1+\frac{11.613890045331\dots}{1.90291648559998\dots}
=7.10320533413700172750577342281\dotsc.
$$


\end{document}